\documentclass[psamsfonts,11pt]{amsart}

\usepackage[margin=1in]{geometry}  
\usepackage{graphicx}              
\usepackage{amsmath}               
\usepackage{amsfonts}              
\usepackage{amsthm}                
\usepackage{verbatim}              
\usepackage{indentfirst}           
\usepackage{underscore}
\usepackage{enumitem}
\usepackage{mathtools} 
\usepackage{amssymb}
\usepackage[all]{xy}
\usepackage{caption}
\usepackage{setspace}
\usepackage[hyperfootnotes=false]{hyperref}

\newtheorem{thm}{Theorem}[section]
\newtheorem{lem}[thm]{Lemma}
\newtheorem{prop}[thm]{Proposition}
\newtheorem*{definition*}{Definition}

\newtheorem{conj}[thm]{Conjecture}
\newtheorem{question}[thm]{Question}
\theoremstyle{remark}
\newtheorem*{rmk}{Remark}
\theoremstyle{remark}

\DeclareMathOperator{\Gal}{Gal}
\DeclareMathOperator{\GL}{GL}

\newcommand{\Z}{\mathbb{Z}}
\newcommand{\Q}{\mathbb{Q}}
\newcommand{\C}{\mathbb{C}}
\newcommand{\fr}{\frac}

\newcommand{\CC}{\mathbb{C}}

\newcommand{\Mod}[1]{\ (\textup{mod}\ #1)}     
\newcommand{\QQ}{\mathbb{Q}}

\usepackage{amssymb,fge}

\title{The $ABC$-Conjecture implies uniform bounds on dynamical Zsigmondy sets}
\author[Nicole R. Looper]{Nicole R. Looper}
\address{Department of Mathematics, Northwestern University; 2033 Sheridan Road, Evanston, IL 60208, USA}
\email{nlooper@math.northwestern.edu}
\begin{document} 
	
	\begin{abstract} \normalsize We prove that the $abc$-Conjecture implies upper bounds on Zsigmondy sets that are uniform over families of unicritical polynomials over number fields. As an application, we use the $abc$-Conjecture to prove that there exist uniform bounds on the index of the associated arboreal Galois representations.
		
	\end{abstract}
	
	\maketitle
	\renewcommand{\thefootnote}{}
	\footnote{Research partially supported by an NSF Graduate Research Fellowship.}
	
	\section{Introduction}
	
	Let $K$ be a number field, with $\mathcal{O}_K$ its ring of integers. Let $f\in K[x]$, and let $\alpha\in K$. Denote the $n$-th iterate of $f$ by $f^n$. For $n\ge2$, we say a prime $\mathfrak{p}$ of $\mathcal{O}_K$ is a \textit{primitive prime divisor} of $f^n(\alpha)$ if $f^n(\alpha)\ne 0$, $v_\mathfrak{p}(f^n(\alpha))>0$, and $v_\mathfrak{p}(f^m(\alpha))\le 0$ for all $1\le m<n$ such that $f^m(\alpha)\ne0$. If $f^n(\alpha)$ does not have any primitive prime divisors, then we say that $n$ is in the \textit{Zsigmondy set} of the forward orbit $\{f^i(\alpha)\}_{i\ge 0}$. Zsigmondy sets have been studied extensively in \cite{GNT}, \cite{Hindes1}, \cite{IngramSilverman}, \cite{IngramSilverman2}, \cite{Krieger}, and \cite{Silverman3}. In \cite{GNT}, the authors show that if $f$ is not dynamically ramified, then the $abc$-Conjecture implies that either $\alpha$ satisfies $f^i(\alpha)=f^j(\alpha)$ for some $i\ne j$, or $f^n(\alpha)$ has a multiplicity one primitive prime divisor for all but finitely many $n$. In another direction, Krieger proves unconditionally in \cite{Krieger} that the cardinality of the Zsigmondy set associated to the critical orbit of $f_c(x)=x^d+c\in \QQ[x]$ is bounded above uniformly over all $c\in \QQ$ such that $f_c$ is not post-critically finite.
	
	In this article, we address the problem of finding bounds on the Zsigmondy sets of unicritical polynomials $f(x)=(x-\gamma)^d+c\in\mathcal{O}_K[x]$, for a given number field $K$ and degree $d\ge 2$. It is easy to show that there does not exist a uniform bound on the sizes of these Zsigmondy sets across all such maps.  For example, given any such $f$ and any $\alpha\in K$ with infinite forward orbit under $f$, if we let $M(z)=z-f^k(\alpha)$ for any $k$, then $MfM^{-1}$ has $k$ in its Zsigmondy set. To treat this, we introduce a quantity $\nu(f)$ that measures how the maximal height of the coefficients of $f$ compares to the height of $f$ in moduli space. In the quadratic case, we assume a standard Height Uniformity Conjecture. This conjecture is a consequence of Vojta's Conjecture \cite{Ih1}.  We remark that Vojta's Conjecture in fact implies the $abc$-Conjecture \cite{Vojta}.
	
	\begin{thm}{\label{thm:main}}
		If $d\ge 3$, assume the $abc$-Conjecture for $K$; if $d=2$, assume further the Height Uniformity Conjecture (see Conjecture \ref{conj:htunif}) and that $K=\QQ$ or $K$ is an imaginary quadratic field. Let $P_d=\{f(x)=(x-\gamma)^d+c\in K[x]\mid c-\gamma\in\mathcal{O}_K, c\ne 0\}$. For $f(x)\in P_d$, let \[\nu(f)=\fr{h(\gamma)}{\max\{1,h(c-\gamma)\}},\] where $h$ denotes the Weil logarithmic height. There exist positive constants $D_1, D_2$ depending only on $d$ and on $K$ such that for all $\alpha\in K$ having infinite forward orbit under $f$, there is a multiplicity 1 primitive prime divisor of $f^n(\alpha)$ for all $n>D_1\log^+(\nu(f))+D_2$. 
	\end{thm}
	
	In particular, given any number field $K$ and $d\ge 3$, there is a uniform bound on the sizes of Zsigmondy sets associated to $\{f^i(\alpha)\}_{i\ge 0}$, where $f(x)=x^d+c\in\mathcal{O}_K$ and $\alpha\in K$ has infinite forward orbit under $f$.
	
	Theorem \ref{thm:main} is partly motivated by an application to dynamical Galois theory. Let $f\in K[x]$ of degree $d\ge 2$ be such that $f^n(x)$ has $d^n$ distinct roots in $\overline{K}$ for all $n\ge1$. Consider the \textit{pre-image tree} associated to $f$, whose vertices are given by \[T=\bigsqcup_{n\ge0}f^{-n}(0),\] and where two vertices $y\in f^{-n}(0)$ and $z\in f^{-n+1}(0)$ are connected by an edge if and only if $f(y)=z$. As $f$ is defined over $K$, the absolute Galois group $\Gal(\overline{K}/K)$ acts on $T$ by bijections, and preserves edge connectivity relations. When $f$ is unicritical and $K$ contains a primitive $d$-th root of unity, $\Gal(\overline{K}/K)$ acts by cyclic permutations on the roots of $f(x)-\alpha$ for any $\alpha\in T$. One thereby obtains a representation \[\rho_f: \Gal(\overline{K}/K)\to[C_d]^{\infty},\] where $C_d$ denotes a cyclic permutation group generated by a $d$-cycle, and $[C_d]^{\infty}$ is the infinite iterated wreath product of $C_d$. The image of this representation has been a major object of study in arithmetic dynamics \cite{Hamblen,Hindes2,Jones2}. One source of motivation for this interest is the following analogy with Serre's open image theorem in number theory. Let $E/K$ be an elliptic curve. If $l\in\mathbb{Z}$ is a prime, the inverse limit of the $l^n$-torsion subgroups of $E$ forms the $l$-adic Tate module $T_l(E)$. If $E/K$ has complex multiplication, then the representation \[\rho_{E,l}: \Gal(\overline{K}/K)\to \textup{Aut}(T_l(E))\cong \GL_2(\Z_l)\] has infinite index image \cite{Silverman2}. Serre's open image theorem addresses the case where $E/K$ does not have complex multiplication. \begin{thm}
		[Serre, \cite{Serre}] Let $E$ be an elliptic curve over a number field $K$ without complex multiplication. Then the representation $\rho_{E,l}$ arising from the Galois action on the $l$-adic Tate module has finite index image in $GL_2(\Z_l)$.\end{thm} This leads to the question of whether a uniform bound exists on this index: in other words, if $G_l(E)$ denotes the image of $\rho_{E,l}$, is there an $N$ such that $[\textup{Aut}(T_l(E)):G_l(E)]\le N$ for all elliptic curves $E/K$ without complex multiplication?  The answer to this question is not yet known.
	
	In the dynamical context, the map $f\in K[x]$ plays the role of multiplication by $l$. If $f$ is unicritical and post-critically finite (PCF), then one can show that the image of $\rho_f$ has infinite index in $[C_d]^{\infty}$. Otherwise, one might pose a similar question as for elliptic curves. 
	
	\begin{question}{\label{question:uniformity}}
		Let $d\ge 2$, and let $K$ be a number field containing a primitive $d$-th root of unity. Does there exist an $N$ such that for all degree $d$, non-PCF, unicritical polynomials $f(x)\in K[x]$ irreducible over $K$, the image $G_K(f)$ of the representation \[\rho_f: \Gal(\overline{K}/K)\to[C_d]^{\infty}\] has index at most $N$? 
	\end{question} We remark that the irreducibility of $f(x)$ over $K$ rules out the obvious counterexamples to such a claim. To date, the only known approach to such a question is to bound the index $n$ such that $K(f^{-n}(0))/K(f^{-n+1}(0))$ fails to have the maximal possible degree, which is $d^{d^{n-1}}$. This question in turn hinges on a study of the critical orbit of $f$. Under mild assumptions on $f$, it suffices to show that for all sufficiently large $n$, $f^n(\gamma)$ has a primitive prime divisor of appropriate multiplicity, where $\gamma$ is the unique finite critical point of $f$. Proving Galois uniformity as in Question \ref{question:uniformity} then requires one to produce uniform bounds on Zsigmondy sets.
	
	As a corollary of the proof of Theorem \ref{thm:main}, we deduce a theorem concerning the Galois uniformity of unicritical polynomials over a given number field $K$. Call $f(x)\in K[x]$ \textit{stable over} $K$ if $f^n(x)$ is irreducible over $K$ for all $n\ge1$. Let $\nu(f)$ be as in the statement of Theorem \ref{thm:main}.
	
	\begin{thm}{\label{thm:introgalunif}} Let $d\ge2$, and let $K$ be a number field containing the $d$-th roots of unity. If $d\ge 3$, assume the $abc$-Conjecture for $K$; if $d=2$, assume further the Height Uniformity Conjecture and that $K=\QQ$ or $K$ is an imaginary quadratic field. Let $\tau\ge 0$. Then the index $[[C_d]^{\infty}:G_K(f)]$ is uniformly bounded over all stable, non-PCF maps $f(x)=(x-\gamma)^d+c\in K[x]$ such that $\nu(f)\le\tau$.
		
	\end{thm} Note that Theorem \ref{thm:introgalunif} does not require $c-\gamma\in\mathcal{O}_K$, in contrast to Theorem \ref{thm:main}. In Section \ref{section:gal}, we exhibit an infinite family of quadratic non-PCF polynomials $\{f_j\}_{j\ge2}$ stable over $K=\QQ$ having the property that for each $j$, the extension $K(f_j^{-j}(0))/K(f_j^{-j+1}(0))$ fails to have the maximal possible degree $2^{2^{j-1}}$. This demonstrates that proving an affirmative answer to Question \ref{question:uniformity} would require both a bound on the number of extensions $K(f^{-n}(0))/K(f^{-n+1}(0))$ failing to be maximal, as well as a bound on the defect \[\fr{d^{d^{n-1}}}{[K(f^{-n}(0)):K(f^{-n+1}(0))]}.\] On the other hand, proving a negative answer to Question \ref{question:uniformity} would likely require showing that the above defect is unbounded across some family.
	\newline
	
	\indent \textbf{Acknowledgements}. I would like to thank Laura DeMarco for helpful discussions regarding this article, and Wade Hindes for introducing me to questions of Galois uniformity through his thesis \cite{Hindes2}. 
	
	\section{Preliminaries}
	
	We begin by introducing notation and basic definitions. If $K/\QQ$ is a number field, and $\mathfrak{p}$ is a finite prime of $K$ with residue field $k_{\mathfrak{p}}$, we let \[N_{\mathfrak{p}}=\fr{\log(\#k_{\mathfrak{p}})}{[K:\Q]}.\] For any $\alpha\in K^*$, the \textit{height} of $\alpha$ is defined as \begin{equation}{\label{eqn:height}} h(\alpha)=-\sum_{\textup{primes }\mathfrak{p}\textup{ of }\mathcal{O}_K}\min\{v_\mathfrak{p}(\alpha),0\}N_\mathfrak{p}+\dfrac{1}{[K:\Q]}\sum_{\sigma:K\hookrightarrow\C}\max\{\log|\sigma(\alpha)|,0\},\end{equation} where $v_\mathfrak{p}$ is the standard $\mathfrak{p}$-adic valuation. (Note that we do not identify complex embeddings $\sigma:K\hookrightarrow\C$ in any way.) Set $h(0)=0$. For a number field $K$, let $M_K$ denote the set of distinct absolute values of $K$ extending those on $\QQ$. Let $M_K^\infty$ denote the set of archimedean places of $K$. Then the height $h(\alpha)$ of $\alpha\in K^*$ can alternatively be expressed as \begin{equation}{\label{eqn:normalhteqn}}h(\alpha)=\fr{1}{[K:\Q]}\sum_{v\in M_K} [K_v:\QQ_v]\log\max\{1,|\alpha|_v\}.\end{equation}
	If $\alpha\in K^*$, then from the product formula, we have the inequality \begin{equation}{\label{eqn:divisortoheight}}\sum_{v_{\mathfrak{p}}(\alpha)>0}v_\mathfrak{p}(\alpha)N_\mathfrak{p}\le h(\alpha).\end{equation} We will use this inequality repeatedly in \S \ref{section:abc} and \S \ref{section:quadratic}. 
	
	The \textit{canonical height} attached to $f\in \overline{\Q}[x]$ is by definition \[\hat{h}_f(\alpha)=\lim_{n\to\infty}\fr{1}{d^n}h(f^n(\alpha))\] for $\alpha\in \overline{\Q}$. It satisfies the transformation rule \[\hat{h}_f(f^n(\alpha))=d^n\hat{h}_f(\alpha).\] The following is a standard lemma about heights which will be essential to later proofs. 
	
	\begin{lem}{\label{lem:triangle}}\cite[Proposition 1.5.15]{BG}
		Let $K$ be a number field. For any $\alpha_1,\alpha_2,\dots,\alpha_n\in K$, \[h(\alpha_1+\dots+\alpha_n)\le \log n + h(\alpha_1)+\dots+h(\alpha_n).\] \end{lem}
	
	\begin{proof}
		For $v\in M_K$, choose $i_v$ to satisfy $\max\{|\alpha_1|_v,\dots,|\alpha_n|_v\}=|\alpha_{i_v}|_v$. Since \[|\alpha_1+\dots+\alpha_n|_v\le \epsilon_v|\alpha_{i_v}|_v,\] where $\epsilon_v=n$ if $v$ is archimedean, and $\epsilon_v=1$ otherwise, it follows that \[\log\max\{1,|\alpha_1+\dots+\alpha_n|_v\}\le\log\epsilon_v+\log\max\{1,|\alpha_{i_v}|_v\}\le \log\epsilon_v +\sum_{i=1}^n \log\max\{1,|\alpha_i|_v\}.\] Noting that \[\fr{1}{[K:\QQ]}\sum_{v\in M_K} [K_v:\QQ_v]\log\epsilon_v=\fr{\log n}{[K:\QQ]}\sum_{v\in M_K^\infty} [K_v:\QQ_v]=\log n\] and applying (\ref{eqn:normalhteqn}) completes the proof.
	\end{proof}
	
	\begin{lem}{\label{lem:ineq1}} Let $d\ge2$, and let $f(x)=x^d+c\in \overline{\QQ}[x]$. For any $\alpha\in \overline{\QQ}$, \[|h(\alpha)-\hat{h}_f(\alpha)|\le \dfrac{1}{d-1}(h(c)+\textup{log}(2)).\]
		
	\end{lem}
	
	\begin{proof} By Lemma \ref{lem:triangle}, we have \[|dh(\alpha)-h(f(\alpha))|\le h(c)+\textup{log}2.\] Taking a telescoping sum, we obtain \begin{equation*}\begin{split}\lim_{n\to\infty} \left|h(\alpha)-\frac{1}{d^n}h(f^n(\alpha))\right| & =\lim_{n\to\infty} \sum_{j=0}^{n-1} \fr{1}{d^j}\left|f^j(\alpha)-\fr{1}{d}f^{j+1}(\alpha)\right| \\ & \le \lim_{n\to\infty} \sum_{j=0}^{n-1} \fr{1}{d^j} [(1/d)(h(c)+\textup{log}2)]=\fr{1}{d-1}(h(c)+\textup{log}2).\end{split}\end{equation*}
		
	\end{proof}
	
	\begin{lem}{\label{lem:mincan2}} Fix $d\ge 2$, and let $K$ be a number field. There exists a $\kappa>0$ depending only on $d$ and on $K$ such that \[h(f^n(\alpha))\ge \kappa d^n \max\{1,h(c)\}-\dfrac{1}{d-1}(h(c)+\textup{log}2)\] for all $f(x)=x^d+c\in\mathcal{O}_K[x]$ and all $\alpha\in K$ having infinite forward orbit under $f$. \end{lem}
	
	\begin{proof} Let $f(x)=x^d+c\in\mathcal{O}_K[x]$, and let $\beta\in K$ be such that $\hat{h}_f(\beta)\ne 0$, and $\hat{h}_f(\beta)\le \hat{h}_f(\alpha)$ for all $\alpha\in K$ such that $\hat{h}_f(\alpha)\ne 0$. Then \[\hat{h}_f(f^n(\beta))\le \hat{h}_f(f^n(\alpha))\] for all $n\ge 1$ and all $\alpha\in K$ such that $\hat{h}_f(\alpha)\ne 0$. From Theorem 1 of \cite{Ingram}, we know that there exists a $\kappa(d,K)>0$ such that \begin{equation}{\label{eqn1}}\hat{h}_{f}(\beta)\ge \kappa \max\{1,h(c)\}.\end{equation} Moreover, by Lemma \ref{lem:ineq1}, \[|h(f^n(\alpha))-\hat{h}_f(f^n(\alpha))|\le \dfrac{1}{d-1}(h(c)+\textup{log}2),\] and thus \[h(f^n(\alpha))-\hat{h}_f(f^n(\beta))\ge -\dfrac{1}{d-1}(h(c)+\textup{log}2)\] for all $\alpha\in K$ having infinite forward orbit under $f$. As $\hat{h}_f(f^n(\beta))=d^n\hat{h}_f(\beta)$, (\ref{eqn1}) implies that \[h(f^n(\alpha))\ge d^n\cdot \kappa\max\{1,h(c)\}-\dfrac{1}{d-1}(h(c)+\textup{log}2)\] for all $\alpha\in K$ having infinite forward orbit under $f$.
	\end{proof}
	
	\begin{rmk}
		The assumptions that $f(x)$ is monic and that $c-\gamma\in\mathcal{O}_K$ in Theorem \ref{thm:main} are necessary solely as a result of the fact that the proof of Theorem \ref{thm:main} relies on Lemma \ref{lem:lb}; Lemma \ref{lem:lb} in turn requires the map $f(x)$ to be of this form in order for the constant $\kappa$ to depend only on $d$ and $K$. We further note that Lemma \ref{lem:lb} (and hence, Theorem 1 of \cite{Ingram}) constitutes a crucial ingredient in the proof of every result in \S \ref{section:abc} and \S \ref{section:quadratic}, with the exception of Lemma \ref{lem:decomposition}. 
		
		The additional requirement that $c\ne 0$ in Theorem \ref{thm:main} arises because it is needed for Proposition \ref{prop:rad} to hold.
	\end{rmk}
	
	\begin{lem}{\label{lem:lb}} Let $f(x)=x^d+c\in\overline{\Q}[x]$, with $d\ge 2$. For any $\alpha\in\overline{\Q}$, \[h(f^n(\alpha))> d^n\left(h(\alpha)-\dfrac{2}{d-1}\max\{1,h(c)\}\right).\]
	\end{lem}
	
	\begin{proof} From Lemma \ref{lem:triangle}, we have \[h(f(\alpha))\ge |dh(\alpha)-h(c)|-\textup{log}2\] and hence \[h(f(\alpha))-dh(\alpha)\ge -2\max\{h(c),\log2\}\ge -2\max\{1,h(c)\}\] for all $\alpha\in\overline{\Q}$. Next, observe that \begin{equation*}\begin{split}\dfrac{1}{d^n} h(f^n(\alpha))-h(\alpha)& =\sum_{j=0}^{n-1} \dfrac{1}{d^j} \left(\dfrac{1}{d} h(f^{j+1}(\alpha))-h(f^j(\alpha))\right) \\ & \ge \left( \sum_{j=0}^{n-1} \dfrac{1}{d^j}\right) \left(\dfrac{-2}{d}\right)\max\{1,h(c)\} \\ & >\left( \sum_{j=0}^{\infty} \dfrac{1}{d^j}\right) \left(\dfrac{-2}{d}\right)\max\{1,h(c)\} \\ & =\dfrac{-2}{d-1} \max\{1,h(c)\}.\end{split}\end{equation*} Therefore \[h(f^n(\alpha))> d^n\left(h(\alpha)-\dfrac{2}{d-1}\max\{1,h(c)\}\right).\] \end{proof} We now address the upper bound on $h(f^n(\alpha))$.
	
	\begin{lem}{\label{lem:upperbound}} Let $d\ge 2$, and let $f(x)=x^d+c\in\overline{\Q}[x]$, where $d\ge 2$. For any $\alpha\in\overline{\Q}$, \[h(f^n(\alpha))< \frac{d^n}{d-1}\left(\textup{log}2+h(c)\right)+d^nh(\alpha).\]\end{lem}
	
	\begin{proof} By Lemma \ref{lem:triangle}, we have \[h(f(\alpha))-dh(\alpha)\le \textup{log}2+h(c).\] Thus \begin{equation*}\begin{split}\frac{1}{d^n}h(f^n(\alpha))-h(\alpha)& =\sum_{j=0}^{n-1} \fr{1}{d^j}\left(\fr{1}{d}h(f^{j+1}(\alpha))-h(f^j(\alpha))\right) \\ & \le \sum_{j=0}^{n-1} \fr{1}{d^{j+1}} \left(\textup{log}2+h(c)\right) \\ & <\sum_{j=0}^{\infty} \fr{1}{d^{j+1}} \left(\textup{log}2+h(c)\right)=\frac{1}{d-1}(\textup{log}2+h(c)).\end{split}\end{equation*}  
		Therefore we obtain \[h(f^n(\alpha))< \frac{d^n}{d-1}\left(\textup{log}2+h(c)\right)+d^nh(\alpha).\]
		
	\end{proof}
	
	\section{Height bounds for unicritical polynomials}
	
	To prove Theorems \ref{thm:lineartau} and \ref{thm:quadraticVojta}, we formulate generalizations of Lemmas \ref{lem:mincan2}, \ref{lem:lb}, and \ref{lem:upperbound} to the case when the finite critical point is nonzero. The proofs follows quickly from the $\gamma=0$ case; we have separated the statements for ease of reading. As previously, let $d\ge 2$, let $K$ be a number field with $\mathcal{O}_K$ its ring of integers, and let $P_d=\{f(x)=(x-\gamma)^d+c\in K[x]\mid c-\gamma\in\mathcal{O}_K,c\ne 0\}$.
	
	\begin{lem}{\label{lem:mincan2'}} There exists a $\kappa>0$ depending only on $d$ and on $K$ such that \[h(f^n(\alpha))\ge \kappa d^n \max\{1,h(c-\gamma)\}-\dfrac{1}{d-1}(h(c-\gamma)+\textup{log}2)-h(\gamma)-\textup{log}2\] for all $f(x)\in P_d$ and all $\alpha\in K$ having infinite forward orbit under $f$. \end{lem}
	
	\begin{proof} Let $\tilde{f}(x)=x^d+c-\gamma$, so that $f^n(\alpha)=\tilde{f}^n(\alpha-\gamma)+\gamma$. From Lemma \ref{lem:mincan2}, we have \[h(\tilde{f}^n(\alpha-\gamma))\ge \kappa d^n\max\{1,h(c-\gamma)\}-\fr{1}{d-1}(h(c-\gamma)+\log2).\] Applying Lemma \ref{lem:triangle} finishes the proof. 
	\end{proof}
	
	\begin{lem}{\label{lem:lb'}} For any $f(x)=(x-\gamma)^d+c\in \overline{\Q}[x]$, and any $\alpha\in\overline{\Q}$, \begin{equation}{\label{abcline5}}h(f^n(\alpha))> d^n\left(h(\alpha-\gamma)-\dfrac{2}{d-1}\max\{1,h(c-\gamma)\}\right)-h(\gamma)-\textup{log}2.\end{equation}
	\end{lem}
	
	\begin{proof} From Lemma \ref{lem:lb}, we know that \[h(\tilde{f}^n(\alpha-\gamma))> d^n(h(\alpha-\gamma)-\fr{2}{d-1}\max\{1,h(c-\gamma)\}.\] Lemma \ref{lem:triangle} again yields the result.
	\end{proof}
	
	\begin{lem}{\label{lem:upperbound'}} For any $f(x)=(x-\gamma)^d+c\in \overline{\Q}[x]$, and any $\alpha\in\overline{\Q}$, \[h(f^n(\alpha))<\frac{d^{n}}{d-1}\left(\textup{log}2+h(c-\gamma)\right)+d^nh(\alpha-\gamma)+h(\gamma)+\textup{log}2.\]
		
	\end{lem}
	
	\begin{proof} Lemma \ref{lem:upperbound} implies \[h(\tilde{f}^n(\alpha-\gamma))< \fr{d^n}{d-1}(\log2+h(c-\gamma))+d^nh(\alpha-\gamma).\] By Lemma \ref{lem:triangle}, we obtain the desired result.
	\end{proof}
	
	\section{The $abc$-Conjecture and primitive prime divisors}{\label{section:abc}}
	
	In this section, we introduce the $abc$-Conjecture for number fields, and use it to prove Proposition \ref{prop:rad}. We define several pieces of notation. Recalling our definition of $h(\alpha)$ for $\alpha\in K$ from (\ref{eqn:height}), we extend this definition to an $n$-tuple. For $(z_1,\dots,z_n)\in K^n\backslash \{(0,\dots,0)\}$ with $n\ge 2$, let \[h(z_1,\dots,z_n)=\sum_{\textup{primes }\mathfrak{p} \textup{ of } \mathcal{O}_K} \min\{v_{\mathfrak{p}}(z_1),\dots,v_{\mathfrak{p}}(z_n)\}N_{\mathfrak{p}}+\dfrac{1}{[K:\QQ]} \sum_{\sigma:K\hookrightarrow\CC} \max\{\textup{log}|\sigma(z_1)|,\dots,\textup{log}|\sigma(z_n)|\}.\] Note that this definition of the height agrees with the one given in (\ref{eqn:height}) if we write $\alpha\in K$ in projective coordinates as $(\alpha,1)$. For any $(z_1,\dots,z_n)\in (K^*)^n$, $n\ge 2$, we define \[I(z_1,\dots,z_n)=\{\textup{primes }\mathfrak{p} \textup{ of } \mathcal{O}_K\mid v_{\mathfrak{p}}(z_i)\ne v_{\mathfrak{p}}(z_j)\textup{ for some } 1\le i,j\le n\}\] and let \[\textup{rad}(z_1,\dots,z_n)=\sum_{\mathfrak{p}\in I(z_1,\dots,z_n)} N_{\mathfrak{p}}.\] The $abc$-Conjecture for a number field $K$ is as follows.
	
	\begin{conj} For any $\epsilon>0$, there exists a constant $C_{K,\epsilon}>0$ such that for all $a,b,c\in K^*$ satisfying $a+b=c$, we have \[h(a,b,c)< (1+\epsilon)(\textup{rad}(a,b,c))+C_{K,\epsilon}.\]
		
	\end{conj} As in the introduction, let $P_d=\{f(x)=(x-\gamma)^d+c\in K[x]\mid c-\gamma\in\mathcal{O}_K, c\ne 0\}$. 
	
	\begin{prop}{\label{prop:rad}} Assume the $abc$-Conjecture for $K$. Fix $d\ge 2$, and let $\tau\ge 0$.  For every $\epsilon>0$, there exists some $N_1=N_1(d,\tau,K,\epsilon)$ such that for all $f(x)\in P_d$ with $\nu(f)\le\tau$, and for all $\alpha\in K$ with infinite forward orbit under $f$, \[\sum_{v_{\mathfrak{p}}(f^n(\alpha))>0} N_{\mathfrak{p}} > (d-1-\epsilon) h(f^{n-1}(\alpha))\] for all $n\ge N_1$. \end{prop}
	
	\begin{proof} Let $f(x)=(x-\gamma)^d+c\in P_d$ with $\nu(f)\le\tau$, and let $\alpha\in K$ have infinite forward orbit under $f$. As $c\ne 0$, the $abc$-Conjecture for $K$ implies that for any given $\epsilon_1>0$, \begin{equation}{\label{abcline1}}(1-\epsilon_1)h((x-\gamma)^d+c,(x-\gamma)^d,c)\le \sum_{\mathfrak{p}\in I(c,(x-\gamma)^d+c,(x-\gamma)^d)} N_{\mathfrak{p}}\end{equation} if $x\in K$, $x-\gamma,(x-\gamma)^d+c\in K^*$, and the left-hand side is sufficiently large. We would like to show that there exists an $N_{\tau,\epsilon_1}$ such that \begin{equation}{\label{eqn:onetermheight}}(1-\epsilon_1)h((x-\gamma)^d+c)\le\sum_{\mathfrak{p}\in I(c,(x-\gamma)^d+c,(x-\gamma)^d)} N_{\mathfrak{p}}\end{equation} when $x=f^{n-1}(\alpha)$ and $n\ge N_{\tau,\epsilon_1}$. To do this, we first claim that for any $\epsilon_2>0$, there exists an $N_{\tau,\epsilon_2}$ depending only on $\tau,\epsilon_2,d$, and $K$ such that \begin{equation}{\label{eqn:c}}\max\{1,h(c)\}\le\epsilon_2h((x-\gamma)^d+c),\end{equation} \begin{equation}{\label{eqn:gamma}}h(\gamma)\le\epsilon_2h((x-\gamma)^d+c),\end{equation} and \begin{equation}{\label{eqn:alpha}}h(\alpha)\le\epsilon_2h((x-\gamma)^d+c)\end{equation} if $x=f^{n-1}(\alpha)$ and $n\ge N_{\tau,\epsilon_2}$. (We will actually only be using (\ref{eqn:c}) in proving (\ref{eqn:onetermheight}), and will use (\ref{eqn:gamma}) and (\ref{eqn:alpha}) later in the proof.) By Lemma \ref{lem:mincan2'}, \begin{equation}{\label{abcline4}}h(f^n(\alpha))\ge \kappa d^n\max\{1,h(c-\gamma)\}-\fr{1}{d-1}(h(c-\gamma)+\log2)-h(\gamma)-\log2.\end{equation} There are two cases:
		\begin{enumerate} \item $h(\alpha-\gamma)\ge \fr{4}{d-1}\max\{1,h(c-\gamma)\}$
			\item $h(\alpha-\gamma)< \fr{4}{d-1}\max\{1,h(c-\gamma)\}$.
		\end{enumerate} In Case (i), we have from (\ref{abcline5}) that \begin{equation}{\label{abcline10}}h(f^n(\alpha))\ge \fr{d^n}{2}h(\alpha-\gamma)-\fr{\tau}{2}h(\alpha-\gamma)-\log2.\end{equation} Moreover, \begin{equation}{\label{abcline7}}\begin{split}h(c)\le h(\gamma)+h(c-\gamma)+\log2 & \le \tau\max\{1,h(c-\gamma)\}+h(c-\gamma)+\log2\\ & \le \fr{(\tau+1)(d-1)}{4}h(\alpha-\gamma)+\log2.\end{split}\end{equation} Comparing (\ref{abcline10}) and (\ref{abcline7}) completes the proof of (\ref{eqn:c}) in Case (i).
		
		In Case (ii), one obtains from (\ref{abcline4}) that \begin{equation}{\label{abcline9}}\begin{split}h(f^n(\alpha))& \ge \kappa d^n\max\{1,h(c-\gamma)\}-\fr{1}{d-1}(h(c-\gamma)+\log2)-\tau\max\{1,h(c-\gamma)\}-\log2 \\ & >\kappa \left(d^n-\fr{2}{d-1}\right)\max\{1,h(c-\gamma)\}-\tau\max\{1,h(c-\gamma)\}-\log2,\end{split}\end{equation} where the second inequality follows from the fact that \[\log2+h(c-\gamma)<2\max\{1,h(c-\gamma)\}.\] Combining the first line of (\ref{abcline7}) with the second line of (\ref{abcline9}) proves the claim in (\ref{eqn:c}).  A similar proof holds for (\ref{eqn:gamma}) and (\ref{eqn:alpha}). 
		
		In projective coordinates, \[((x-\gamma)^d+c,c,(x-\gamma)^d)=\left(\fr{(x-\gamma)^d+c}{c},1,\fr{(x-\gamma)^d}{c}\right)\] for any $x\in K$. We now compare the height of this point in $\mathbb{P}^2(K)$ to that of $((x-\gamma)^d+c,1,1)\in\mathbb{P}^2(K)$. One checks by the triangle inequality that for any $x\in K$ and any complex embedding $\sigma: K\hookrightarrow\CC$, \[\log\max\{|\sigma((x-\gamma)^d+c)|,1\}-\log\max\left\{\left|\sigma\left(\dfrac{(x-\gamma)^d+c}{c}\right)\right|,1,\left|\sigma\left(\dfrac{(x-\gamma)^d}{c}\right)\right|\right\}\le\log\max\{1,|\sigma(c)|\},\] and for any finite prime $\mathfrak{p}$ of $\mathcal{O}_K$, \[\min\{v_\mathfrak{p}((x-\gamma)^d+c)),0\}-\min\left\{v_\mathfrak{p}\left(\dfrac{(x-\gamma)^d+c}{c}\right),0,v_\mathfrak{p}\left(\dfrac{(x-\gamma)^d}{c}\right)\right\}\le \max\{0,v_\mathfrak{p}(c)\}.\] Combining these two inequalities and summing over all places, we obtain:
		
		\[h((x-\gamma)^d+c,1,1)-h\left(\fr{(x-\gamma)^d+c}{c},1,\fr{(x-\gamma)^d}{c}\right)\le 2h(c).\] Applying (\ref{eqn:c}), therefore, we see from (\ref{abcline1}), (\ref{abcline10}), and (\ref{abcline9}) that for any $\epsilon_1>0$, there exists an $N_{\tau,\epsilon_1}$ such that \begin{equation*}(1-\epsilon_1)h((x-\gamma)^d+c)\le\sum_{\mathfrak{p}\in I(c,(x-\gamma)^d+c,(x-\gamma)^d)} N_{\mathfrak{p}}\end{equation*} when $x=f^{n-1}(\alpha)$ and $n\ge N_{\tau,\epsilon_1}$. (By (\ref{eqn:c}), if $x=f^{n-1}(\alpha)$ and $n\ge N_{\tau,\epsilon_1}$, then $x-\gamma\ne 0$ and $(x-\gamma)^d+c\ne 0$, as required in order for (\ref{abcline1}) to hold.) 
		
		As $c-\gamma\in\mathcal{O}_K$, the only primes $\mathfrak{p}$ such that $v_{\mathfrak{p}}(f^n(\alpha))<0$ are those such that $v_{\mathfrak{p}}(\alpha)<0$ or $v_\mathfrak{p}(c)<0$. Let $x=f^{n-1}(\alpha)$. We have  \begin{equation}{\label{abcline2}}\begin{split}\sum_{\mathfrak{p}\in I(c,(x-\gamma)^d+c,(x-\gamma)^d)} N_{\mathfrak{p}} &\le \sum_{v_{\mathfrak{p}}(c(x-\gamma)^d(x-\gamma)^d+c))>0} N_{\mathfrak{p}}+h(\alpha)+h(c) \\ & \le \sum_{v_{\mathfrak{p}}(c)>0} N_{\mathfrak{p}}+\sum_{v_{\mathfrak{p}}((x-\gamma)^d)>0} N_{\mathfrak{p}} + \sum_{v_{\mathfrak{p}}((x-\gamma)^d+c)>0} N_{\mathfrak{p}}+h(\alpha)+h(c) \\ &
		\le 2h(c)+h(x-\gamma)+h(\alpha)+\sum_{v_{\mathfrak{p}}((x-\gamma)^d+c)>0} N_{\mathfrak{p}}\end{split}\end{equation} when $n\ge N_{\tau,\epsilon_1}$; here the third inequality follows from (\ref{eqn:divisortoheight}). We also have \[h((x-\gamma)^d+c)\ge |dh(x-\gamma)-h(c)|-\textup{log}(2).\] Therefore, (\ref{eqn:onetermheight}) and (\ref{abcline2}) yield \[\sum_{v_{\mathfrak{p}}(f^n(\alpha))>0} N_{\mathfrak{p}} > (1-\epsilon_1)(dh(f^{n-1}(\alpha)-\gamma)-h(c)-\textup{log}2)-2h(c)-h(\alpha)-h(f^{n-1}(\alpha)-\gamma)\] for all $n\ge N_{\tau,\epsilon_1}$. Using Lemma \ref{lem:triangle}, this implies \begin{equation}{\label{abcline12}}\begin{split}\sum_{v_{\mathfrak{p}}(f^n(\alpha))>0} N_{\mathfrak{p}} & > d(1-\epsilon_1)(h(f^{n-1}(\alpha)-h(\gamma)-\log2)-h(f^{n-1}(\alpha))-h(\gamma)-\log2\\&-(1-\epsilon_1)(h(c)-\log2)-2h(c)-h(\alpha) \\ & = (d-d\epsilon_1-1)h(f^{n-1}(\alpha))-(d-d\epsilon_1+1)(h(\gamma)+\log2)\\ &-(1-\epsilon_1)(h(c)-\log2)-2h(c)-h(\alpha)\end{split}\end{equation} for all $n\ge N_{\tau,\epsilon_1}$. From (\ref{eqn:c}), (\ref{eqn:gamma}), and (\ref{eqn:alpha}), it follows that for any $\epsilon_3>0$, there exists an $N_{\tau, \epsilon_3}$ depending only on $\tau,\epsilon_3,d$ and $K$ such that \begin{equation}{\label{abcline13}}(d-d\epsilon_1+1)(h(\gamma)+\log2)-(1-\epsilon_1)(h(c)-\log2)-2h(c)-h(\alpha)\le \epsilon_3h(f^{n-1}(\alpha))\end{equation} for all $n\ge N_{\tau,\epsilon_3}$. Choosing $\epsilon_1$ and $\epsilon_3$ such that $d\epsilon_1+\epsilon_3<\epsilon$, and taking $N_1=\max\{N_{\tau,\epsilon_1}, N_{\tau,\epsilon_3}\}$, (\ref{abcline12}) and (\ref{abcline13}) together imply the statement of Proposition \ref{prop:rad}.\end{proof} The proof of the next proposition is similar to that of Proposition 5.1 in \cite{GNT}. \begin{prop}{\label{prop:imprimitive}} Let $d\ge 2$, and let $\delta>0$. There exists an $N_{\tau,\delta}$ such that for all $f(x)\in P_d$ with $\nu(f)\le\tau$, and all $\alpha\in K$ having infinite forward orbit under $f$, \[\sum_{\mathfrak{p}\in Z}N_{\mathfrak{p}}\le \delta h(f^n(\alpha))\] for all $n\ge N_{\tau,\delta}$, where $Z$ is the set of finite primes $\mathfrak{p}$ in $\mathcal{O}_K$ such that $v_\mathfrak{p}(f^n(\alpha))>0$, $v_{\mathfrak{p}}(f^i(\alpha))>0$ and $f^i(\alpha)\ne 0$ for some $i<n$.
		
	\end{prop}
	
	\begin{proof} Since $f^n(\alpha)=f^{n-k}(f^k(\alpha))$, observe that if $v_\mathfrak{p}(f^k(\alpha))>0$ and $v_\mathfrak{p}(f^n(\alpha))>0$ for some prime $\mathfrak{p}$ of $\mathcal{O}_K$, then $v_\mathfrak{p}(f^{n-k}(0))>0$. It follows that if $f^n(\alpha)\ne 0$, then \[\begin{split}\sum_{\mathfrak{p}\in Z} N_{\mathfrak{p}} & \le \sum_{k=0}^{\lfloor n/2\rfloor} h(f^k(\alpha))+h(f^k(0)) \\ & < \sum_{k=0}^{\lfloor n/2\rfloor} h(\gamma)+\log2+\fr{d^k}{d-1}(\log2+h(c-\gamma))+d^kh(\alpha-\gamma) \\ & +\sum_{k=0}^{\lfloor n/2\rfloor}h(\gamma)+\log2+ \fr{d^k}{d-1}(\log2+h(c-\gamma))+d^kh(\gamma) \\ & \le \sum_{k=0}^{\lfloor n/2\rfloor}2\log2+\fr{2d^k}{d-1}(\log2+h(c-\gamma))+(d^k+2)h(\gamma)+d^kh(\alpha-\gamma)\\ & \le 2d^{\lfloor n/2\rfloor+1}(\log2+h(c-\gamma))+d^{\lfloor n/2\rfloor+1}(h(\gamma)+h(\alpha-\gamma))+(\lfloor n/2\rfloor+1)(2\log2) \\ & < 4d^{\lfloor n/2\rfloor+1}\max\{1,h(c-\gamma)\}+d^{\lfloor n/2\rfloor+1}(\tau\max\{1,h(c-\gamma)\}+h(\alpha-\gamma))+2\log2(\lfloor n/2\rfloor+1),\end{split}\] where the first inequality follows from (\ref{eqn:divisortoheight}) and the second inequality follows from Lemma \ref{lem:upperbound'}.  As in the proof of Proposition \ref{prop:rad}, we divide the proof into two cases:
		\begin{enumerate}
			\item $h(\alpha-\gamma)\ge\fr{4}{d-1}\max\{1,h(c-\gamma)\}$
			\item  $h(\alpha-\gamma)<\fr{4}{d-1}\max\{1,h(c-\gamma)\}$.
		\end{enumerate}
		In Case (i), we have \[\sum_{\mathfrak{p}\in Z} N_\mathfrak{p}\le (d-1)d^{\lfloor n/2\rfloor +1}h(\alpha-\gamma)+d^{\lfloor n/2\rfloor+1}h(\alpha-\gamma)\left(\fr{\tau(d-1)}{4}+1\right)+2\log2(\lfloor n/2\rfloor+1),\] assuming $f^n(\alpha)\ne 0$. The lower bound given in (\ref{abcline10}) then leads to the desired result in Case (i). In Case (ii), the upper bound on $\sum_{\mathfrak{p}\in Z}N_\mathfrak{p}$ given above, combined with (\ref{abcline9}), completes the proof.\end{proof}
	
	\section{Height Uniformity Conjecture applied to the quadratic case}{\label{section:quadratic}}
	
	In this section, we use a Height Uniformity Conjecture to address the case where $f$ is quadratic. The version of the Height Uniformity Conjecture we cite follows from Vojta's Conjecture \cite{Ih1}. For a version involving integral points on curves of genus at least one, see \cite{Ih2}.
	
	Let $d\ge 2$, let $K$ be a fixed number field, and let $P_d=\{f(x)=(x-\gamma)^d+c\in K[x]\mid c-\gamma\in\mathcal{O}_K, c\ne 0\}$. For $f(x)=a_dx^d+a_{d-1}x^{d-1}+\dots+a_1x+a_0\in\overline{\Q}[x]$, let $h(f)=\max\{h(a_i)\}_{0\le i\le d}$. (Note that this is somewhat different from the usual definition.) We first prove three lemmas that will be needed to prove Theorem \ref{thm:main} when $f$ is quadratic.
	
	\begin{lem}{\label{lem:heightgrowth}}
		Fix $d\ge 2$, and $i\ge 1$. For every $\tau\ge 0$ and every $C_1,C_2\in\mathbb{R}$, there exists an $N_{\tau,1}$ such that for all $f(x)\in P_d$ with $\nu(f)\le\tau$, \begin{equation}{\label{eqn:lemheightgrowth}}h(f^n(\alpha))>C_1h(f^i)+C_2\end{equation} for all $n\ge N_{\tau,1}$ whenever $\alpha\in K$ has infinite forward orbit under $f$.
		
	\end{lem}
	
	\begin{proof}
		Let $f(x)\in P_d$ be such that $\nu(f)\le\tau$, and let $\alpha\in K$ have infinite forward orbit under $f$. By Lemma \ref{lem:triangle}, we have \begin{equation*}h(f^i)\le \kappa_1\max\{h(\gamma),h(c)\}+\kappa_2\end{equation*} for some constants $\kappa_1,\kappa_2$ that depend only on $i$ and $d$. Therefore, there exists a constant $\kappa_3(d,i)$ such that \begin{equation*}h(f^i)\le \kappa_3\max\{1,h(\gamma),h(c)\}.\end{equation*} Combining this with (\ref{eqn:c}) and (\ref{eqn:gamma}) finishes the proof.\end{proof}
	
	\begin{rmk}
		From (\ref{abcline10}) and (\ref{abcline9}), we see that for any given choices of $C_1$ and $C_2$, the $d^n$ needed for (\ref{eqn:lemheightgrowth}) to hold grows linearly in $\nu(f)$; thus, $N_{\tau,1}$ grows linearly in $\log^+\tau$. We will make use of this fact in the proof of Theorem \ref{thm:quadraticVojta}.
	\end{rmk}

	\begin{lem}{\label{lem:heightsqueeze}}
		Let $d=2$, and let $\epsilon>0$. For every $\tau\ge 0$, there exists an $N_{\tau,\epsilon}$ such that for all $f(x)\in P_2$ with $\nu(f)\le\tau$, \[(1-\epsilon)h(f^n(\alpha))\le 2h(f^{n-1}(\alpha))\le (1+\epsilon)h(f^n(\alpha))\] for all $n\ge N_{\tau,\epsilon}$ and all $\alpha\in K$ having infinite forward orbit under $f$.
	\end{lem}
	
	\begin{proof}
		Let $f(x)\in P_2$, and let $\alpha\in K$ have infinite forward orbit under $f$. For all $x\in\overline{\mathbb{Q}}$, and all degree $d$ polynomials $g\in\overline{\mathbb{Q}}[x]$, we have from \cite[Theorem 3.11]{Silverman} that \begin{equation}{\label{eqn:heightsqueeze}}h(g(x))=dh(x)+O_d(h(g)+1).\end{equation} However, Lemma \ref{lem:heightgrowth} implies that for every $B>1$, there exists an $N_{\tau,B}$ (depending on $B$, $\tau$, and $K$) such that \[h(f^n(\alpha))\ge B(h(f)+1)\] whenever $n\ge N_{\tau,B}$. Letting $B$ be sufficiently large, we conclude from taking $d=2$ and $f=g$ in (\ref{eqn:heightsqueeze}) that \[(1-\epsilon)h(f^n(\alpha))\le 2h(f^{n-1}(\alpha))\le (1+\epsilon)h(f^n(\alpha))\] for all $n\ge N_{\tau,C}$.
	\end{proof} In order to prove Proposition \ref{prop:heightunif}, we also require the following lemma, which is similar to Lemma 2.2 of \cite{Hindes1}. For a number field $K$ and a finite set of primes $S$ of $\mathcal{O}_K$ containing the archimedean places, let $\mathcal{O}_{K,S}=\{\alpha\in K:v_\mathfrak{p}(\alpha)\ge0, \mathfrak{p}\notin S\}$, and let $\mathcal{O}_{K,S}^*$ denote the unit group of $\mathcal{O}_{K,S}$. \begin{lem}{\label{lem:decomposition}}
		Let $K$ be a number field, let $\alpha\in K$, and let $f(x)=(x-\gamma)^d+c\in K[x]$, with $d\ge2$. Let $l\ge2$, and let $S$ be the minimal set of primes of $\mathcal{O}_K$ such that $S$ contains the archimedean places, $S$ contains each finite prime $\mathfrak{p}$ of $\mathcal{O}_K$ where $\alpha$,$\gamma$, or $c$ has negative valuation, and $\mathcal{O}_{K,S}$ is a UFD. For every $n\ge 1$, there is a decomposition \begin{equation}{\label{eqn:decomposition}}f^n(\alpha)=u_nd_ny_n^l, \hspace{1mm}\text{ for some } y_n\in\mathcal{O}_{K,S},u_n\in\mathcal{O}_{K,S}^*,d_n\in\mathcal{O}_K\end{equation} satisfying the following properties:
		
		(1) $0\le v_\mathfrak{p}(d_n)\le l-1$ for all $\mathfrak{p}\notin S$
		
		(2) $0\le v_\mathfrak{p}(d_n)<h_K$ for all $\mathfrak{p}\in S$, where $h_K$ is the class number of $K$
		
		(3) The height $h(u_n)$ satisfies $h(u_n)\le C(l-1)^2(h(\alpha)+h(\gamma)+h(c))$ for all $n$, where $C$ is some constant depending only on $K$.
	\end{lem}
	
	\begin{proof}
		Since $f^n(\alpha)\in\mathcal{O}_{K,S}$, we can write $f^n(\alpha)=u_nd_ny_n^l$, where $d_n,y_n\in\mathcal{O}_{K,S}$ and $u_n\in\mathcal{O}_{K,S}^*$. We can also assume that $0\le v_\mathfrak{p}(d_n)\le l-1$ for all $\mathfrak{p}\notin S$; we do this by writing \[d_n=p_1^{e_1}p_2^{e_2}\cdots p_s^{e_s}(p_1^{f_1}p_2^{f_2}\cdots p_s^{f_s})^l\] where the $p_i$ are primes in $\mathcal{O}_{K,S}$, and the $e_i,f_i$ are integers such that $v_{p_i}(d_n)=f_il+e_i$ and $0\le e_i<l$. Replacing $d_n$ with $(p_1^{e_1}p_2^{e_2}\cdots p_s^{e_s})$ and $y_n$ with $(y_np_1^{f_1}p_2^{f_2}\cdots p_s^{f_s})$, we can assume $0\le v_p(d_n)\le l-1$ for all $\mathfrak{p}\notin S$.
		
		Now let $\mathfrak{p}_i\in S$. There exists an $a_i\in\mathcal{O}_K$ and $h_K\ge n_i\ge1$ such that $\mathfrak{p}_i^{n_i}=(a_i)$. Therefore, writing $v_{\mathfrak{p}_i}(d_n)=f_in_i+r_i$ for some $0\le r_i<n_i$ and setting $d_n'=d_n/(\prod_i a_i^{f_i})$, we get $0\le v_\mathfrak{p}(d_n')=v_\mathfrak{p}(d_n)\le l-1$ for all $\mathfrak{p}\notin S$ and $v_{\mathfrak{p}_i}(d_n')=r_i<h_K$ for all $\mathfrak{p}_i\in S$. Substituting $d_n'$ for $d_n$ and $u_n(\prod_i a_i^{f_i})$ for $u_n$, conditions (1) and (2) are both met. Note also that now $d_n\in\mathcal{O}_K$, as required in (\ref{eqn:decomposition}).
		
		Turning to (3), let $S'$ be the minimal set of primes of $\mathcal{O}_K$ containing the archimedean places of $K$ such that $\mathcal{O}_{K,S'}$ is a UFD. Suppose $S$ contains some prime not in $S'$. Absorbing $l$-th powers into $y_n$, we can write \[u_n=q_1^{s_1}q_2^{s_2}\cdots q_r^{s_r},\] where the $q_i$ are pairwise non-associate prime elements of $\mathcal{O}_{K,S'}$ and $0\le s_i\le l-1$. Then each $q_i\in\mathfrak{q}_i$ for a unique prime ideal $\mathfrak{q}_i$ of $\mathcal{O}_K$ not in $S'$. Write \[q_i=v_1^{r_1}v_2^{r_2}\cdots v_t^{r_t}\] for some basis $\{v_j\}_{j=1}^t$ of $\mathcal{O}_{K,S'\cup\{\mathfrak{q}_i\}}$; we can assume $0\le r_i\le l-1$, by absorbing $l$-th powers into $y_n$. Let $\Delta_{K/\Q}$ denote the discriminant of $K$, let $r_K=\textup{rank}(\mathcal{O}_{K,S'}^*)$, and let $s$ be the number of complex places of $K$, with conjugate places identified. Let \[D=\fr{1}{2}\log|\Delta_{K/\Q}|+s\log(2/\pi),\] and let $m_{q_i}$ be the maximal norm of a finite place contained in $S'\cup\{\mathfrak{q}_i\}$. By Theorem 6.2 of \cite{Lenstra}, the basis $\{v_j\}_{j=1}^t$ can be chosen so that \[h(v_j)\le D+\log(m_{q_i})\] for all $j$. But we can bound the right-hand side from above by $B\log(N_{K/\Q}(\mathfrak{q}_i))$, where $B$ is some constant depending only on $K$. Since $\textup{rank}(\mathcal{O}_{K,S'\cup\{\mathfrak{q}_i\}}^*)=r_K+1$, this implies \[h(q_i)=h(v_1^{r_1}v_2^{r_2}\cdots v_t^{r_t})\le (l-1)B\log(N_{K/\Q}(\mathfrak{q}_i))(r_K+1).\] Hence \begin{equation*}\begin{split}h(u_n) & \le B(l-1)^2(r_K+1)\sum_{i=1}^r \log(N_{K/\Q}(\mathfrak{q}_i)) \\ & \le B(l-1)^2(r_K+1)[K:\Q]\sum_{\mathfrak{p}\in S\backslash S'} N_\mathfrak{p} \\ & \le B(l-1)^2(r_K+1)[K:\Q](h(\alpha)+h(\gamma)+h(c)) \end{split}\end{equation*} Taking $C=B(r_K+1)[K:\Q]$ completes the proof when $S-S'\ne\emptyset$. A similar argument proves the lemma when $S=S'$.

	\end{proof}
	\begin{definition*}
		For a given $n\ge1$, let $Y$ be the set of primes $\mathfrak{p}$ in $\mathcal{O}_K$ such that $v_{\mathfrak{p}}(f^n(\alpha))>0$, and let $Y_1$, $Y_2$ denote the subset of $Y$ consisting of multiplicity $1$ and $2$ divisors of $f^n(\alpha)$, respectively. Let $Y_{3+}$ denote the set of primes in $Y$ dividing $f^n(\alpha)$ to multiplicity at least $3$.\end{definition*} We now introduce the Height Uniformity Conjecture.
	
	\begin{conj}{\label{conj:htunif}} Let $K$ be a number field. For each $d\ge5$, there exist positive constants $C_1$ and $C_2$ such that for all $F\in K[x]$ of degree $d$ with $\textup{disc}(F)\ne0$, if $x,y\in K$ satisfy $y^2=F(x)$, then \[h(x)\le C_1h(F)+C_2.\]
	\end{conj}
	
	\begin{prop}{\label{prop:heightunif}} Let $K=\QQ$ or let $K$ be an imaginary quadratic field. Assume the $abc$-Conjecture for $K$ and the Height Uniformity Conjecture. Let $\tau\ge0$. Then for any sufficiently small $\epsilon>0$, there exists an $N=N(\tau,\epsilon,K)$ such that for all $f(x)\in P_2$ with $\nu(f)\le\tau$, \[\sum_{\mathfrak{p}\in Y_1}N_{\mathfrak{p}}>\epsilon h(f^n(\alpha))\] for all $\alpha\in K$ having infinite forward orbit under $f$, and for all $n\ge N$.
	\end{prop}
	\begin{proof} Let $f(x)\in P_2$ with $\nu(f)\le\tau$, and let $\alpha\in K$. Write $f^n(\alpha)=u_nd_ny_n^2$ as in (\ref{eqn:decomposition}) with $l=2$. Then for $n\ge 3$, $f^{n-3}(\alpha)$ is the $x$-coordinate of a $K$-rational point on the affine curve \[u_nd_nY^2=f^3(X),\] or equivalently, on the curve \[Y^2=u_nd_nf^3(X)\] over $K$. By the Height Uniformity Conjecture, we have \[h(x)\le C_1(h(u_nd_n)+h(f^3))+C_2\le 2C_1\max\{h(u_nd_n),h(f^3)\}+C_2.\] On the other hand, by Lemma \ref{lem:heightgrowth}, there exists an $N_{\tau,1}$ such that for all $f(x)\in P_d$ with $\nu(f)\le\tau$, \[h(f^{n-3}(\alpha))>2C_1h(f^3)+C_2\] for all $n\ge N_{\tau,1}$ whenever $\alpha$ has infinite forward orbit under $f$. Thus, when $n\ge N_{\tau,1}$, \[h(f^{n-3}(\alpha))\le 2C_1h(u_nd_n)+C_2,\] with $h(u_nd_n)>0$. From Northcott's Theorem, it follows that for all such $f(x)\in P_d$ and $\alpha\in K$, there is some $1>\delta_1>0$ such that \begin{equation}{\label{eqn:htunif}}h(f^{n-3}(\alpha))\le \fr{1}{\delta_1}h(u_nd_n)\end{equation} for all $n\ge N_{\tau,1}$. As $K=\QQ$ or $K$ is an imaginary quadratic field, and we can assume $d_n\in\mathcal{O}_K$ is nonzero, we have \[h(d_n)=\sum_{v_\mathfrak{p}(d_n)>0}v_\mathfrak{p}(d_n)N_\mathfrak{p}.\] From Lemma \ref{lem:decomposition}, therefore, one obtains: \begin{equation*}\begin{split}h(u_nd_n) \le h(u_n)+h(d_n)&\le  C(h(\alpha)+h(\gamma)+h(c))+\sum_{\substack{v_\mathfrak{p}(f^n(\alpha))\textup{ odd} \\ \mathfrak{p}\notin S}} N_\mathfrak{p}+h_K\sum_{\mathfrak{p}\textup{ finite in } S}N_\mathfrak{p} \\ & \le h_K\sum_{\mathfrak{p}\textup{ finite in }S'}N_\mathfrak{p}+(C+h_K)(h(\alpha)+h(\gamma)+h(c))+\sum_{v_\mathfrak{p}(f^n(\alpha))\textup{ odd}} N_\mathfrak{p}\end{split}\end{equation*} Let $\epsilon_2>0$ be such that $\epsilon_2<\delta_1/2$. By Lemma \ref{lem:heightgrowth}, and by (\ref{eqn:c}),(\ref{eqn:gamma}), and (\ref{eqn:alpha}), there exists an $N_{\tau,\epsilon_2}$ such that \[h_K\sum_{\mathfrak{p}\textup{ finite in }S'}N_\mathfrak{p}+(C+h_K)(h(\alpha)+h(\gamma)+h(c))\le\epsilon_2 h(f^{n-3}(\alpha))\] for all $n\ge N_{\tau,\epsilon_2}$. Applying (\ref{eqn:htunif}), we see that \[\sum_{\mathfrak{p}\in Y_1\cup Y_{3+}}N_\mathfrak{p}+\epsilon_2 h(f^{n-3}(\alpha))\ge\delta_1 h(f^{n-3}(\alpha))\] and hence by our choice of $\epsilon_2$, \begin{equation}{\label{eqn:oddchunk}}\sum_{\mathfrak{p}\in Y_1\cup Y_{3+}}N_\mathfrak{p}\ge\fr{\delta_1}{2}h(f^{n-3}(\alpha))\end{equation} for all $n\ge\max\{N_{\tau,1},N_{\tau,\epsilon_2}\}$. Choose a positive $\epsilon\le \delta_1/216$; we note for future use that for this choice of $\epsilon$, we automatically have \begin{equation}{\label{eqn:2/5}}\fr{1/2-3\epsilon}{1+\epsilon}>\fr{2}{5}\end{equation} since $\delta_1<1$. Then by (\ref{eqn:oddchunk}), for each $n\ge\max\{N_{\tau,1},N_{\tau,\epsilon_2}\}$, either \[(i)\hspace{3mm} \sum_{\mathfrak{p}\in Y_{3+}}N_\mathfrak{p}\ge54\epsilon h(f^{n-3}(\alpha))\] or \[(ii)\hspace{3mm}\sum_{\mathfrak{p}\in Y_1}N_\mathfrak{p}\ge54\epsilon h(f^{n-3}(\alpha)).\] Suppose that some $n\ge\max\{N_{\tau,1},N_{\tau,\epsilon_2}\}$ satisfies (i). Note that by Lemma \ref{lem:heightsqueeze}, there exists an $N_{\tau,2}$ such that \[h(f^{n-3}(\alpha))> \fr{1}{9} h(f^n(\alpha))\] if $n\ge N_{\tau,2}$. Thus, from the Case (i) condition, we obtain \begin{equation}{\label{eqn:case6eps}}\sum_{\mathfrak{p}\in Y_{3+}}N_{\mathfrak{p}}\ge 6\epsilon\sum_{\mathfrak{p}\in Y}N_{\mathfrak{p}}\ge6\epsilon\sum_{\mathfrak{p}\in Y_2}N_\mathfrak{p}\end{equation} if $n\ge\max\{N_{\tau,1},N_{\tau,\epsilon_2},N_{\tau,2}\}$. By Lemma \ref{lem:heightsqueeze} along with Proposition \ref{prop:rad}, there exists an $N_{\tau,\epsilon}$ such that if $n\ge N_{\tau,\epsilon}$, we have \begin{equation}{\label{eqn:radbound}}(1-\epsilon)h(f^n(\alpha))\le 2\sum_{\mathfrak{p}\in Y} N_{\mathfrak{p}}.\end{equation} This gives \begin{equation*}(1-\epsilon)\sum_{\mathfrak{p}\in Y_1} N_{\mathfrak{p}}+(2-2\epsilon)\sum_{\mathfrak{p}\in Y_2}N_{\mathfrak{p}}+(3-3\epsilon)\sum_{\mathfrak{p}\in Y_{3+}}N_{\mathfrak{p}}\le 2\sum_{\mathfrak{p}\in Y} N_{\mathfrak{p}}\end{equation*} if $n\ge \max\{N_{\tau,1},N_{\tau,\epsilon_2},N_{\tau,2},N_{\tau,\epsilon}\}$. In this case, (\ref{eqn:case6eps}) reduces to \[(1-\epsilon)\sum_{\mathfrak{p}\in Y_1}N_{\mathfrak{p}}+(2+\epsilon)\sum_{\mathfrak{p}\in Y_2}N_{\mathfrak{p}}+(5/2-3\epsilon)\sum_{\mathfrak{p}\in Y_{3+}}N_{\mathfrak{p}}\le 2\sum_{\mathfrak{p}\in Y}N_{\mathfrak{p}}.\] Subtracting $2\sum_{\mathfrak{p}\in Y} N_\mathfrak{p}$ from each side, we get \[(-1-\epsilon)\sum_{\mathfrak{p}\in Y_1} N_\mathfrak{p} +\epsilon\sum_{\mathfrak{p}\in Y_2} N_\mathfrak{p}+(1/2-3\epsilon)\sum_{\mathfrak{p}\in Y_{3+}} N_\mathfrak{p}\le 0.\] Simplifying and applying (\ref{eqn:radbound}), one obtains \begin{equation}{\label{eqn:Vojta3}}\sum_{\mathfrak{p}\in Y_1}N_{\mathfrak{p}}\ge \fr{1/2-3\epsilon}{1+\epsilon}\sum_{\mathfrak{p}\in Y_{3+}}N_{\mathfrak{p}}\ge \fr{1/2-3\epsilon}{1+\epsilon}\cdot6\epsilon \sum_{\mathfrak{p}\in Y}N_{\mathfrak{p}}>3\epsilon(1-\epsilon)\left(\fr{1/2-3\epsilon}{1+\epsilon}\right) h(f^n(\alpha)).\end{equation} By Lemma \ref{lem:heightsqueeze}, we can assume that for all $n\ge N_{\tau,\epsilon}$, and all $\alpha\in K$ with infinite forward orbit under $f$, we have $h(f^n(\alpha))>h(f^{n-3}(\alpha))$. Therefore (\ref{eqn:2/5}) and (\ref{eqn:Vojta3}) imply \begin{equation}{\label{eqn:Vojta2}}\sum_{\mathfrak{p}\in Y_1}N_{\mathfrak{p}}>\fr{6}{5}\epsilon(1-\epsilon)h(f^{n-3}(\alpha))\end{equation} if $n\ge\max\{N_{\tau,1},N_{\tau,\epsilon_2},N_{\tau,2},N_{\tau,\epsilon}\}$ satisfies (i). 
		
		On the other hand, if some $n\ge\max\{N_{\tau,1},N_{\tau,\epsilon_2},N_{\tau,2},N_{\tau,\epsilon}\}$ satisfies (ii), then (\ref{eqn:Vojta2}) holds trivially. Therefore \[\sum_{\mathfrak{p}\in Y_1}N_{\mathfrak{p}}\ge\fr{6}{5}\epsilon(1-\epsilon)h(f^{n-3}(\alpha))>\fr{2}{15}\epsilon(1-\epsilon)h(f^n(\alpha))\] for any $n\ge\max\{N_{\tau,\epsilon},N_{\tau,\epsilon_2},N_{\tau,1},N_{\tau,2}\}$. For all sufficiently small choices of positive $\epsilon\le\delta_1/216$, substituting $8\epsilon$ for $\epsilon$ completes the proof.

	\end{proof}
	
	\section{Uniform bounds on Zsigmondy sets}
	
	Throughout this section, let $K$ be a fixed number field, and let $P_d$ be as in \S\ref{section:abc}, \S\ref{section:quadratic}, and the introduction. For convenience, we restate the results of Theorem \ref{thm:main} as Theorems \ref{thm:lineartau} and \ref{thm:quadraticVojta}.
	
	\begin{thm}{\label{thm:lineartau}} Assume the $abc$-Conjecture for $K$. Let $d\ge 3$, and let $f(x)\in P_d$. There exist positive constants $D_1, D_2$ depending only on $d$ and $K$ such that if \[n>D_1\log^+\nu(f)+D_2,\] then $f^n(\alpha)$ has a multiplicity 1 primitive prime divisor for all $\alpha\in K$ having infinite forward orbit under $f$.
		
	\end{thm}
	
	\begin{thm}{\label{thm:quadraticVojta}} Let $f(x)\in P_2$, and suppose $K=\QQ$ or $K$ is an imaginary quadratic field. Assume the $abc$-Conjecture for $K$ and the Height Uniformity Conjecture. There exist positive constants $D_3, D_4$ depending only on $K$ such that if \[n>D_3\log^+\nu(f)+D_4,\] then $f^n(\alpha)$ has a multiplicity 1 primitive prime divisor for all $\alpha\in K$ having infinite forward orbit under $f$.
		
	\end{thm}
	
	Fix $d\ge 2$. For $f(x)\in P_d$, we define $N_f$ to be the least $N$ such that for all $\alpha\in K$ having infinite forward orbit under $f$, there is a multiplicity 1 primitive prime divisor of $f^n(\alpha)$ for all $n\ge N$. (If such an $N_f$ does not exist, set $N_f=\infty$.) By Theorems \ref{thm:lineartau} and \ref{thm:quadraticVojta}, Vojta's Conjecture implies $N_f<\infty$. Our next result reveals that Theorems \ref{thm:lineartau} and \ref{thm:quadraticVojta} give the best possible upper bound on $N_f$, assuming $N_f<\infty$.
	
	\begin{thm}{\label{thm:taunec}} Fix $d\ge 2$.  There exist constants $D_5>0$ and $D_6$ depending only on $d$ such that \[N_f\ge D_5\log^+\nu(f)+D_6\] for all $f(x)\in P_d$.
		
	\end{thm}
	
	\begin{proof}[Proof of Theorem \ref{thm:lineartau}] As $d\ge 3$, we see from Proposition \ref{prop:rad} and Lemma \ref{lem:heightsqueeze} that for all sufficiently small $\epsilon>0$, $N_1=N_1(d,K,\tau,\epsilon)$ satisfies \[\sum_{\mathfrak{p}\in Y_1} N_\mathfrak{p}>\epsilon h(f^n(\alpha))\] for all $n\ge N_1$.  Let $Y_{\textup{prim}}$ denote the set of multiplicity 1 primitive prime divisors of $f^n(\alpha)$. Letting $\delta=\epsilon/2$ in Proposition \ref{prop:imprimitive}, it follows that \begin{equation}{\label{eqn:yprim}}\sum_{\mathfrak{p}\in Y_{\textup{prim}}}N_\mathfrak{p}\ge \fr{\epsilon}{2}h(f^n(\alpha))\end{equation} for all $n\ge\max\{N_1,N_{\tau,\delta}\}$. From the proofs of Propositions \ref{prop:rad} and \ref{prop:imprimitive}, we see that we can choose $N_1$ and $N_{\tau,\delta}$ to grow at most linearly in $\log^+\tau$; in particular, there exists some $M>0$ depending only on $d$ and on $K$ such that if $n\ge3\log^+(\nu(f))+M$, then $f^n(\alpha)$ has a multiplicity 1 primitive prime divisor for any $\alpha\in K$ having infinite forward orbit under $f$. 
		
	\end{proof}
	
	\begin{proof}[Proof of Theorem \ref{thm:quadraticVojta}]
		As noted in the remark following the proof of Lemma \ref{lem:heightgrowth}, the $N_{\tau,1}$ resulting from the proof grows linearly in $\log^+\tau$. For a given choice of $\epsilon_2$, the same holds for the $N_{\tau,\epsilon_2}$ produced in the proof of Proposition \ref{prop:heightunif} using Lemma \ref{lem:heightgrowth}. Lemma \ref{lem:heightsqueeze} follows from Lemma \ref{lem:heightgrowth}, so similarly, for a given choice of $\epsilon$, the  $N_{\tau,2}$ and $N_{\tau,\epsilon}$ produced in the course of the proof of Proposition \ref{prop:heightunif} grow linearly in $\log^+\tau$. Therefore $\max\{N_{\tau,\epsilon},N_{\tau,\epsilon_2},N_{\tau,1},N_{\tau,2}\}$ grows linearly in $\log^+\tau$. Let $\delta=\epsilon/2$, where $\epsilon$ is as in the statement of Proposition \ref{prop:heightunif}, and let $N$ be as in the conclusion of Proposition \ref{prop:heightunif}. Then \[\sum_{\mathfrak{p}\in Y_{\textup{prim}}}N_\mathfrak{p}\ge\fr{\epsilon}{2}h(f^n(\alpha))\] for all $n\ge\max\{N,N_{\tau,\delta}\}$, where $N_{\tau,\delta}$ is as in Proposition \ref{prop:imprimitive}. As previously noted, $N_{\tau,\delta}$ grows linearly in $\log^+\tau$. This completes the proof.
	\end{proof}
	
	\begin{proof}[Proof of Theorem \ref{thm:taunec}] Let $f(x)\in P_d$. Suppose $\gamma,c$ are chosen so that $f$ is non-PCF, and $f^N(\gamma)=0$ for some $N$; we can do this by writing $\tilde{f}(x)=x^d+(c-\gamma)$ where $\tilde{f}$ is non-PCF with $c-\gamma\in\mathcal{O}_K$, setting $\gamma=-\tilde{f}^N(0)$, and considering $f=M\tilde{f}M^{-1}=(x-\gamma)^d+c$, where $M(x)=x+\gamma$. Since $f^N(\gamma)=0$, we see that $N_f>N$. On the other hand, by Lemma \ref{lem:upperbound}, $N$ satisfies \[h(\gamma)=h(\tilde{f}^N(0))< \dfrac{d^N}{d-1}(h(c-\gamma)+\log2),\] so \[\log(d-1)+\log^+\left(\dfrac{h(\gamma)}{h(c-\gamma)+\log2}\right)< N\log d.\] Since \[h(c-\gamma)+\log2<2\max\{1,h(c-\gamma)\},\] one deduces that \[\log^+(\nu(f))-\log2<N\log d,\] and thus, $N_f$ must satisfy \[\fr{1}{\log d}\left(\log^+\nu(f)-\log2\right)< N_f.\] \end{proof}
	
	\section{Applications to Galois uniformity}{\label{section:gal}}
	
	In this section, we apply Theorems \ref{thm:lineartau} and \ref{thm:quadraticVojta} to the question of Galois uniformity of unicritical polynomials. Let $d\ge 2$, and let $K$ be a number field containing a primitive $d$-th root of unity. We say $f(x)\in K[x]$ is \textit{stable over} $K$ if $f^n(x)$ is irreducible over $K$ for all $n\ge 1$. If $f(x)=(x-\gamma)^d+c\in K[x]$, and $K_n=K_n(f)$ denotes the splitting field of $f^n(x)$, then we say $K_n/K_{n-1}$ is \textit{maximal} if $K_n/K_{n-1}$ has degree $d^{d^{n-1}}$. We use a standard lemma relating the maximality of $K_n/K_{n-1}$ to the arithmetic properties of $f^n(\gamma)$.
	
	\begin{lem}[\cite{Hamblen}]{\label{lem:Stoll}} Assume $K$ contains a primitive $d$-th root of unity, and that $f(x)=(x-\gamma)^d+c\in K[x]$ is stable over $K$. Let $n\ge 2$. Then $K_n/K_{n-1}$ is maximal if and only if for all $p\mid d$, $f^n(\gamma)$ is not a $p$-th power in $K_{n-1}$.\end{lem} We will also invoke a lemma about the stability of unicritical polynomials. As we only need to apply it to quadratic maps, we state it for the quadratic case.  The statement in the general case is similar.
	
	\begin{lem}[\cite{Jones1}]{\label{lem:stability}} Suppose $f(x)=(x-\gamma)^2+c\in K[x]$ is irreducible over $K$. If, for all $n\ge 2$, $f^n(\gamma)$ is not a square in $K$, then $f(x)$ is stable over $K$.
		
	\end{lem}
	
	\begin{thm}{\label{thm:galunif}}
		Assume $K$ contains a primitive $d$-th root of unity. If $d\ge 3$, assume the $abc$-Conjecture; if $d=2$, assume further the Height Uniformity Conjecture, and that $K=\QQ$ or $K$ is an imaginary quadratic field. Let $\tau\ge 0$. There exists an $N=N(\tau,d,K)$ such that for all non-PCF maps $f(x)=(x-\gamma)^d+c\in K[x]$ stable over $K$ with $\nu(f)\le\tau$, the extension $K_n/K_{n-1}$ is maximal for all $n\ge N$.  
		
		In particular, for all non-PCF maps $f(x)=(x-\gamma)^d+c\in K[x]$ stable over $K$ with $\nu(f)\le\tau$, the index $[[C_d]^{\infty}:G_K(f)]$ is uniformly bounded.
	\end{thm} One might wonder whether the growth of our bound $N$ in terms of $\tau$ is merely an artefact of the use of Theorems \ref{thm:lineartau} and \ref{thm:quadraticVojta}. In fact, the greatest $n$ such that $K_n/K_{n-1}$ fails to be maximal cannot be bounded uniformly across all stable, non-PCF, unicritical $f(x)\in K[x]$, without taking $\nu(f)$ into account. As an example of this, we produce an infinite family of quadratic, stable, non-PCF maps with the property that the maximal such $n$ of each map is unbounded across the family.
	
	\begin{prop}{\label{prop:example}}
		Let $f(x)=x^2+2$, and let $K=\QQ$. For any $i\ge 2$, let \[f_i(x)=f\left(x+2+\fr{f^i(0)-2}{2}\right)-\left(2+\fr{f^i(0)-2}{2}\right)\in K[x].\] Then $K_i(f_i)/K_{i-1}(f_i)$ is not maximal. 
	\end{prop} 
	
	\begin{proof}
		First we show that each $f_i(x)$ is stable over $\QQ$. Fix an $f_i(x)$, and let $\gamma$ be its unique finite critical point. By Lemma \ref{lem:stability}, it suffices to show that $f_i^n(\gamma)$ is not a square in $\QQ$, and that $f_i(x)$ is irreducible over $\QQ$. We accomplish the former by showing that $v_2(f_i^n(\gamma))=1$ for all $n$. Indeed, since $v_2(f^n(0))=1$ for all $n$, we have \[f^i(0)-2= (f^{i-1}(0))^2\equiv 4\Mod{8}\] for all $i\ge 2$.  Thus \[\fr{f^i(0)-2}{2}\equiv2\Mod{8},\] so \[2+\fr{f^i(0)-2}{2}\equiv4\Mod{8}\] for all $i\ge2$. Since $f_i=MfM^{-1}$ where $M(x)=x-2-\fr{f^i(0)-2}{2}$, and $f^n(0)\equiv2$ or $6\Mod{8}$ for all $n$, this implies that $f_i^n(\gamma)\equiv2$ or $6\Mod{8}$ for all $n$. Thus $v_2(f_i^n(\gamma))=1$. As $2$ is not a square in $\QQ$, we conclude that for all $n\ge1$, $f_i^n(\gamma)$ is not a square in $\QQ$. Turning to the question of whether $f_i(x)$ is irreducible over $\QQ$, we see that $v_2(\textup{Disc}(f_i))=v_2(-4f_i(\gamma))=3$. As $f_i$ is quadratic, this implies that $f_i(x)$ is irreducible over $\QQ$. Therefore, by Lemma \ref{lem:stability}, $f_i(x)$ is stable over $\QQ$. 
		
		Next, we remark that by construction, $f_i(\gamma)=-f_i^i(\gamma)$. Moreover, $-f_i(\gamma)$ is a square in $K_1$, as $\textup{Disc}(f_i)=-4f_i(\gamma)$. But by Lemma \ref{lem:Stoll}, this implies that $K_i/K_{i-1}$ is not maximal, as $f_i^i(\gamma)$ is then a square in $K_{i-1}$.

	\end{proof}
	
	\begin{rmk}
		Proposition \ref{prop:example} serves as an apt illustration of Theorem \ref{thm:taunec}. Since the $f_i$ are all affine conjugate to one another, we see directly that the upper bound on the largest element of the Zsigmondy set (associated to the critical orbit) depends on how $h(f_i)$ compares to the height of $f_i$ in the moduli space of quadratic rational maps. Proposition \ref{prop:example} shows that an answer to Question \ref{question:uniformity} requires more than just a bound on the greatest $n$ such that $K_n/K_{n-1}$ is not maximal. 
	\end{rmk}
	
	Finally, we derive Theorem \ref{thm:galunif} from Theorems \ref{thm:lineartau} and \ref{thm:quadraticVojta}.
	
	\begin{proof}[Proof of Theorem \ref{thm:galunif}] We first formulate the following analogue of Lemma \ref{lem:lb'} for the special case $\alpha=\gamma$. The only difference lies in the assumptions on the coefficients of $f$; here, we do not require $c-\gamma\in\mathcal{O}_K$.
		
		\begin{lem}{\label{lem:mincan2''}} Let $K$ be a number field, and let $d\ge 2$. There exists a $\kappa=\kappa(d,K)>0$ depending only on $d$ and on $K$ such that \[h(f^n(\gamma))\ge\kappa d^n\max\{1,h(c-\gamma)\}-\fr{1}{d-1}(h(c-\gamma)+\log2)-h(\gamma)-\log2\] for all $f(x)=(x-\gamma)^d+c\in K[x]$ such that $\gamma$ has infinite forward orbit under $f$. \end{lem} The proof of Lemma \ref{lem:mincan2''} proceeds similarly to that of Lemma \ref{lem:mincan2'}, with the added observation that from the proof of Theorem 1 of \cite{Ingram}, it follows that if $f(x)=x^d+c\in K[x]$ is non-PCF, then there exists a $\kappa>0$ depending only on $d$ and on $K$ so that  \[\hat{h}_f(0)\ge\kappa\max\{1,h(c)\}.\] Next, we note that as $\alpha=\gamma$, the primes $\mathfrak{p}$ of $\mathcal{O}_K$ with $v_\mathfrak{p}(f^n(\gamma))<0$ must satisfy either $v_\mathfrak{p}(\alpha)=v_\mathfrak{p}(\gamma)<0$ or $v_\mathfrak{p}(c)<0$. But this, along with Lemma \ref{lem:mincan2'}, was the only consequence of having $c-\gamma\in\mathcal{O}_K$ that was used in the proofs of Proposition \ref{prop:rad} and \ref{prop:heightunif} (and thus indirectly, in Theorems \ref{thm:lineartau} and \ref{thm:quadraticVojta}). For the appropriate number fields $K$, therefore, Theorems \ref{thm:lineartau} and \ref{thm:quadraticVojta} hold over all maps of the form $f(x)=(x-\gamma)^d+c\in K[x]$ in the special case $\alpha=\gamma$.
		
		Suppose $d\ge 3$. Let $\epsilon>0$, and let $N_2$ be such that $h(d)<\fr{\epsilon}{4} h(f^n(\gamma))$ and $h(c)+h(\gamma)<\fr{\epsilon}{4}h(f^n(\gamma))$ for all $n\ge N_2$ and all $f(x)=(x-\gamma)^d+c\in K[x]$ with $\nu(f)\le\tau$.  (Such an $N_2$ exists by Lemma \ref{lem:heightgrowth}, and by inequalities (\ref{eqn:c}) and (\ref{eqn:alpha}).) Assume $\epsilon$ is sufficiently small, as was required in the proof of Theorem \ref{thm:lineartau}, let $\delta=\epsilon/2$, and let $N_1$ and $N_{\tau,\delta}$ be as in the proof of Theorem \ref{thm:lineartau}. Then (\ref{eqn:yprim}) applied to $\alpha=\gamma$ reads \[\sum_{\mathfrak{p}\in Y_{\textup{prim}}} N_\mathfrak{p}>h(d)+h(c)+h(\gamma)\] for all $n\ge N=\max\{N_1,N_2,N_{\tau,\delta}\}$. It follows that for all such $n$, $f^n(\gamma)$ has a multiplicity 1 primitive prime divisor $\mathfrak{p}_n$ such that $v_{\mathfrak{p}_n}(d)=v_{\mathfrak{p}_n}(c)=v_{\mathfrak{p}_n}(\gamma)=0$.  From Lemma 2.6 of \cite{Jones1} we have \[\textup{Disc}(f^i)=\pm d^{d^i}(\textup{Disc}(f^{i-1}))(f^i(\gamma))^{d-1}.\] The primes $\mathfrak{q}$ in $\mathcal{O}_K$ for which $f^i$ has some coefficient with negative $\mathfrak{q}$-adic valuation must satisfy either $v_\mathfrak{q}(\gamma)<0$ or $v_\mathfrak{q}(c)<0$. On the other hand, if no coefficient of $f^i$ has negative $\mathfrak{q}$-adic valuation, then $\mathfrak{q}$ ramifies in $K_i$ only when $\mathfrak{q}$ divides $\textup{Disc}(f^i)$. Hence the conditions on $\mathfrak{p}_n$ imply that for all $i\le n-1$, $\mathfrak{p}_n$ does not ramify in $K_i$. From this, we conclude that for all $n\ge N$, $f^n(\gamma)$ is not a $p$-th power in $K_{n-1}$ for any $p\mid d$. By Lemma \ref{lem:Stoll}, the extension $K_n/K_{n-1}$ is maximal for all $n\ge N$. 
		
		The proof in the case $d=2$ follows similarly.
		
	\end{proof}

\end{document}